\documentclass[11pt, oneside]{article}   	% use "amsart" instead of "article" for AMSLaTeX format
\usepackage{geometry}                		% See geometry.pdf to learn the layout options. There are lots.
\usepackage{amsmath, amsthm}     
\usepackage{amsfonts} 
\usepackage{amssymb}
\usepackage{graphicx}				% Use pdf, png, jpg, or eps with pdflatex; use eps in DVI mode
\newtheorem{thm}{Theorem}[section]

\newtheorem{lem}[thm]{Lemma}

\theoremstyle{remark}

\theoremstyle{definition}
\newtheorem{defi}[thm]{Definition}

\title{The Banach-Tarski Paradox}
\author{Katie Buchhorn}

\begin{document}
\maketitle
Supervised by Adam Sierakowski and David Robertson
\subsection*{Abstract}
In 1924, S. Banach and A. Tarski proved an astonishing, yet rather counterintuitive paradox: given a solid ball in $\mathbb{R}^3$, it is possible to partition it into finitely many pieces and reassemble them to form two solid balls, each identical in size to the first.

This paper was written for the fulfilment  of a \emph{summer scholarship} during an undergraduate degree at the University of Wollongong. The aim of the paper is to provide a comprehensive proof of the Banach-Tarski paradox thereby introducing the notions of paradoxical and equidecomposable sets which are phrased in terms of group actions. Once the reader has a firm grasp of these ideas, the proof of the Banach-Tarski Paradox is fairly straight forward, provided we have the Axiom of Choice at our disposal.

\section{Paradoxical Decompositions}

\begin{defi}
A group $G$ is a set together with a binary operation 
\begin{align*}
G \times G \rightarrow& G \\
(g,h) \mapsto& gh
\end{align*}
satisfying the following axioms:
\begin{itemize}
\item[(i)]Closure: if $a$, $b \in G$, then $a  b \in G$.
\item[(ii)]Identity: There is an identity element $e \in G$ such that $e  a = a = a e$, for every $a \in G$.
\item[(iii)]Inverse: There  must exist an inverse for each group element, 

i.e., for each $a \in G$, $\exists a^{-1} \in G$ such that $a  a^{-1} = e = a^{-1}  a$.

\item[(iv)]Associativity: for all $a,b,c \in G$, $(a  b )  c = a ( b  c )$.
\end{itemize}
\end{defi}

The operation with respect to which the group is formed is called the \emph{group} $operation$.

\begin{defi}
Let $G$ be a group and let $X$ be a set. Then a (left) group action of $G$ on $X$ is a function
\begin{equation*}
G \times X \rightarrow X,
\end{equation*}
\begin{equation*}
(g,x) \mapsto g \cdot x
\end{equation*}
that satisfies the following axioms:
\begin{itemize}
\item[(i)] Associativity: $(g h) \cdot x = g \cdot (h \cdot x)$ for all $g,h \in G$ and $x \in X$.

\item[(ii)]Identity: $e \cdot a = a $ for all $x \in X$ (where $e$ denotes the identity element in $G$).
\end{itemize}
\end{defi}

\begin{defi} $G$ be a group acting on a set $X$ and suppose $E \subseteq X$. $E$ is \emph{$G$-paradoxical} if for some positive integers $m,n$ there exists pairwise disjoint subsets $A_1, ... , A_n , B_1, ... , B_m$ of $E$ and $g_1, ... , g_n, h_1, ... , h_m \in G$ such that
\begin{equation*}
E = \bigcup _{i=1}^n g_i\cdot A_i  \text{ and } E = \bigcup _{i=1}^m h_i\cdot B_i 
\end{equation*}
\end{defi}

A good way to interpret this is to picture $E$ containing two disjoint subsets $\bigcup _{i=1}^n A_i$, $\bigcup _{i=1}^m B_i$ each of which can be broken down into their namesakes $A_i$ or $B_i$ and then moved around consecutively via elements of the group $G$ to cover all of $E$.

Note also that there is a stronger alternative; finding subsets $A_1, ... , A_n , B_1, ... , B_m$ of $E$ which in fact \emph{partition} $E$ (The only extra condition being that these subsets initially union to give $E$). This idea will be formally introduced later.

\subsection{Free Groups}
Free groups will be a recurring idea in the progression of the paper, so we will formally introduce the concept. If $S$ is a set, the free group generated by $S$  is the group of all reduced finite words with letters from $\{ s,s^{-1} \mid s\in S \}$. A word is called reduced if it contains no pairs of adjacent letters $ss^{-1}$ or $s^{-1}s$. The group composition is concatenation of words followed by reduction, that is, removing pairs of the mentioned forms. The rank of a group is the size of the generating set.

\begin{thm}
A free group $\mathbb{F}_2$ of rank 2 is $\mathbb{F}_2$-paradoxical, where $\mathbb{F}_2$ acts on itself by left multiplication.
\end{thm}

\begin{proof}
Let $\mathbb{F}_2$ be the free group with the generating set $ \{ a, b \}$ and let $\rho \in \{ a, a^{-1}, b, b^{-1} \}$. Now let $ \Psi( \rho )$ denote the set of elements in $\mathbb{F}_2$ beginning with $\rho$. 
\newpage
\textbf{Figure 1}: The Cayley graph of the free group $\mathbb{F}_2$ on two generators $a$ and $b$

\includegraphics[scale=0.45]{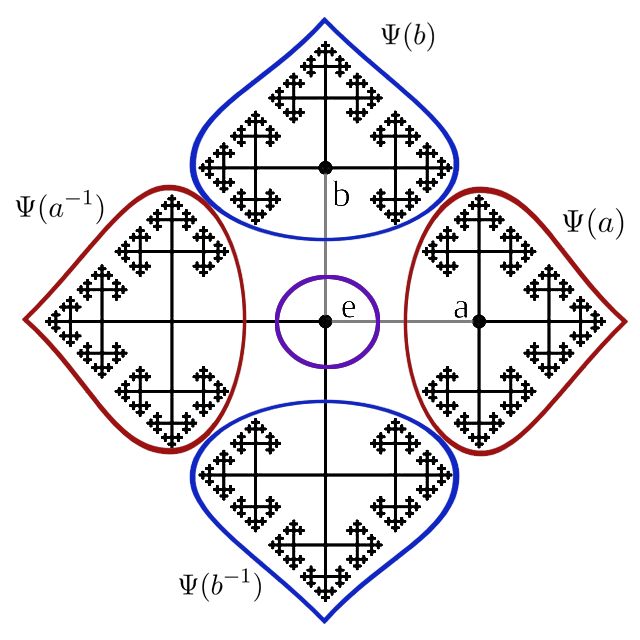}

It can be easily seen from Figure 1 that
\begin{equation}
\mathbb{F}_2 = \{ e \} \bigcup   \Psi( a )  \bigcup   \Psi( a^{-1}  )  \bigcup   \Psi( b )  \bigcup   \Psi( b^{-1}  ) 
\end{equation}
and that these subsets are pairwise disjoint.  We are aiming to find group elements of $\mathbb{F}_2$ that act on $\Psi( a )$ and  $\Psi( a^{-1}  )$ to rearrange the subsets in such a way that their union will give $\mathbb{F}_2$, and similarly for $\Psi( b )$ and  $\Psi( b^{-1}  )$.
\newpage
\textbf{Figure 2} The Cayley graph of the free group $\mathbb{F}_2$ and its subsets

\includegraphics[scale=0.45]{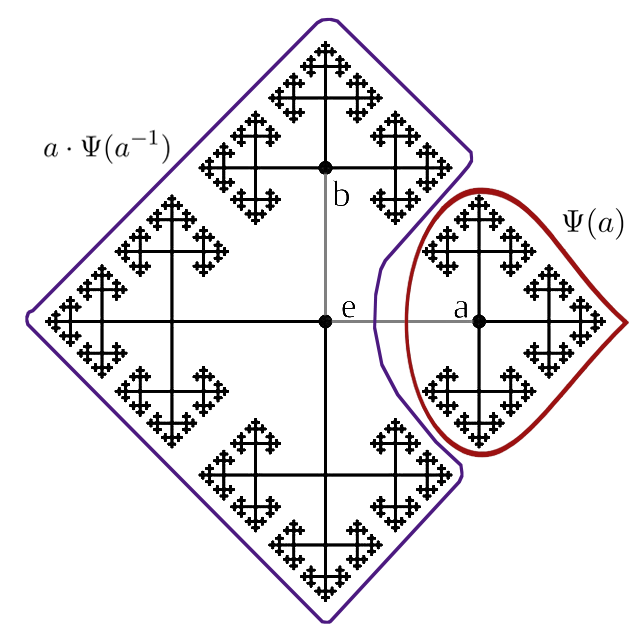}

For any element $ h \in F\setminus \Psi(a)$, this implies that $h \notin \Psi(a)$. So $h$ does not begin with the letter $a$. Then, if $a^{-1}$ acts on $h$ (on the left), there will be no cancellation hence $a^{-1} \cdot h \in \Psi( a^{-1}  )$.

Now,
\begin{equation*}
h = ( a \cdot a^{-1} ) \cdot h = a \cdot ( a^{-1} \cdot h) \in a \cdot \Psi( a^{-1}  )
  \end{equation*}

So for any element in $\mathbb{F}_2 \setminus \Psi(a)$, we can conclude that the element must in fact be in $a \cdot \Psi( a^{-1}  )$. Similarly, any element in $\mathbb{F}_2 \setminus \Psi(b)$ implies that the element is in  $b \cdot \Psi( b^{-1}  )$. So,

\begin{equation*}
a \cdot \Psi( a^{-1}  ) \cup \Psi( a ) = \mathbb{F}_2
 \text{ and }
b \cdot \Psi( b^{-1}  ) \cup \Psi( b ) = \mathbb{F}_2
\end{equation*}

Hence a free group $\mathbb{F}_2$ of rank 2 is $\mathbb{F}_2$-paradoxical.

\end{proof}

Notice that the following choice of subsets would have indeed \emph{partitioned} the free group $\mathbb{F}_2$:\\
$A_1 = \Psi( a^{-1}  ) \cup \{e\} \cup \{a\} \cup \{a^2\} \cup \{a^3\}...$ \\
$A_2 = \Psi(a) \setminus \{e\} \cup \{a\} \cup \{a^2\} \cup \{a^3\}...$\\
$B_1 = \Psi(b^{-1} )$\\
$B_2 = \Psi(b)$\\

with $a \cdot A_1 \cup A_2 = \mathbb{F}_2 $ and $b \cdot B_1 \cup B_2 = \mathbb{F}_2$, offering an alternative proof of Theorem 1.2.
	
\section{Equidecomposibility}
\begin{defi}
Let $G$ be a group acting on a set $X$, and let $A,B \subseteq X$. We say that $A$ and $B$ are \emph{$G$-congruent} if there exists $g \in G$ such that $g \cdot A = B$.
\end{defi}

\begin{defi}
Suppose $G$ acts on $X$ and $A,B	 \subseteq X$. $A$ and $B$ are \emph{$G$-equidecomposable} (denoted $A \sim_G B$) if $A$ and $B$ can each be partitioned into the same finite number of respectively $G$-congruent pieces. Formally written, $A \sim_G B$ if $A = \bigcup _{i=1}^n A_i$ ,	$B = \bigcup _{i=1}^n B_i$ with

\begin{equation*}
A_i \cap A_j = \emptyset = B_i \cap B_j \text{ for } i  < j \le n,
\end{equation*}
	
and there are $g_1, ... , g_n \in G$ such that, for each $i \le n$, $g_i \cdot A_i = B_i$.\\	
\end{defi}

In plain terms this definition means to say that if you can take one subset, $A$, and break it apart into $n$ pieces, rearrange those $n$ pieces via the group action of $G$ to form the $n$ pieces that union together to form another subset $B$, then $A$ and $B$ are $G$-equidecomposable.
	
\begin{lem}
Suppose $G$ acts on $X$ and $E \subseteq X$. If the set $E$ has two disjoint subsets $A$ and $B$ such that $A \sim_G E$ and $B \sim_G E$, then $E$ is $G$-paradoxical. 
\end{lem}
\newpage
\textbf{Figure 3} An illustration of Lemma 2.3

\includegraphics[scale=0.9]{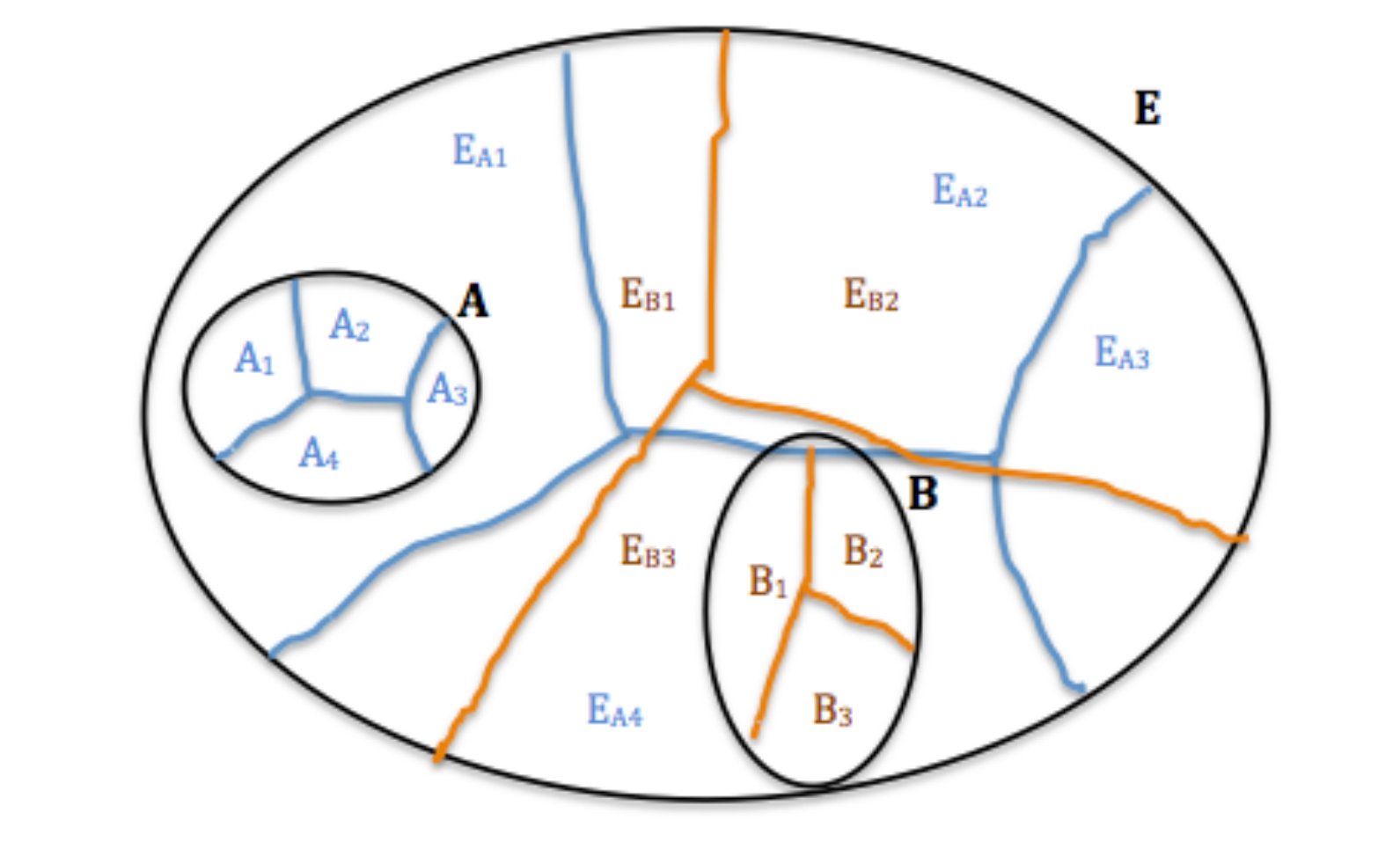}

\begin{proof} We know that, because $A \subseteq E$ and $E$ are $G$-equidecomposable, that $A$ and $E$ can each be partitioned into the same number of respectively $G$-congruent pieces, so
\begin{equation*}
A = \bigcup _{i=1}^n A_i \text{ and } E = \bigcup _{j=1}^n E_j \text{ for some $A_i \subseteq A$ and $E_j \subseteq E$}
\end{equation*}

\begin{equation*}
\text{and, for each $i \le n$, } g_i \cdot A_i = E_i \text{ with } g_i \in G
\end{equation*}
	
Similarly, for $B\subseteq E$ and $E$,
\begin{equation*}
B = \bigcup _{k=1}^m B_k \text{ and } E = \bigcup _{l=1}^m E_l \text{ for some $B_k \subseteq A$ and $E_l \subseteq E$}
\end{equation*}

\begin{equation*}
\text{and, for each $i \le m$, } h_i \cdot A_i = E_i\text{ with } h_i \in G
\end{equation*}

Therefore, 

\begin{equation*}
E = \bigcup _{i=1}^m E_i = \bigcup _{i=1}^m  h_i \cdot B_i \text{ and } E = \bigcup _{i=1}^n E_i = \bigcup _{i=1}^n  g_i \cdot A_i
\end{equation*}

where each $A_1, ..., A_n$ and each $B_1, ..., B_m$ are pairwise disjoint, and $A$ and $B$ are also disjoint, so we can conclude that subsets $A_1, ... , A_n , B_1, ... , B_m$ of $E$ are also pairwise disjoint. Hence $E$ is $G$-paradoxical.
\end{proof}

%\begin{lem}
%Suppose $G$ acts on a set $X$ and $E, E' \subseteq X$ are G-equidecomposable. If $E$ is $G$-paradoxical, so is $E'$.
%\end{lem}

\section{Banach-Schroder-Bernstein Theorem}

This theorem is an interesting concept to visualise, and acts as a platform for further study on equidecomposibility relationships. A detailed proof has been included for the readers' ease.  

\begin{lem}
if $A \sim B$ then there is a bijection $g: A \rightarrow B$ such that $C \sim g(C)$ whenever $C \subseteq A$
\end{lem}

\begin{lem}
if $A_1 \cap A_2 = \emptyset = B_1 \cap B_2$, and if $A_1 \sim B_1$ and $A_2 \sim B_2$, then $A_1 \cup A_2 \sim B_1 \cup B_2$.
\end{lem}

\begin{thm} [Banach-Schroder-Bernstein Theorem]
Suppose $G$ acts on $X$ and $A, B \subseteq X$. If $A$ is G-equidecomposable to a subset of $B$, and if $B$ is G-equidecomposable to a subset of $A$, then $A$ is G-equidecomposable to $B$.
\end{thm}
\begin{proof}
Denote $B_1 \subseteq B_2$ and $A_1 \subseteq A_2$. Define $f: A \rightarrow B_1$ and $g: A_1 \rightarrow B$ as bijections guaranteed by Lemma 3.1 as $A \sim B_1$ and $B \sim A_1$.\\

Let $C_0 = A \setminus A_1$ and define inductively $C_{n+1} = g^{-1}f(C_n)$. Now, let $C= \bigcup^{\infty}_{n=0} C_n$. We claim that 
\begin{equation}
g(A \setminus C) = B \setminus f(C)
\end{equation} To see this, consider $x \in A_1$,
\begin{align*}
g(x) \in B \setminus f(C) &\Leftrightarrow g(x) \notin f(C)\\
&\Leftrightarrow g(x) \notin f(\bigcup^{\infty}_{n=0} C_n)\\
&\Leftrightarrow g(x) \notin \bigcup^{\infty}_{n=0} f(C_n)\\
&\Leftrightarrow g(x) \notin f(C_n) \text{ for all } n\ge 0\\
&\Leftrightarrow g^{-1}g(x) \notin g^{-1}f(C_n) \text{ for all } n\ge 0\\
&\Leftrightarrow x \notin C_{n+1} \text{ for all } n\ge 0\\
&\Leftrightarrow x \notin \bigcup^{\infty}_{n=0} C_{n+1}\\
&\Leftrightarrow x \notin \bigcup^{\infty}_{n=1} C_{n}\\
&\Leftrightarrow x \notin C \text{ by } \ast \\
&\Leftrightarrow x \in A \setminus C\\
&\Leftrightarrow g(x) \in g(A \setminus C)
\end{align*}
$\ast$ provided $x \notin C_0 \Leftrightarrow x \notin  A \setminus A_1$. Which is true because $x \in A_1$. \\

So our claim is true. We now wish to show that $g(A \setminus C) \sim A \setminus C$. We can use Lemma 3.1 but first we need to convince ourselves that $A \setminus C \subseteq A_1$.
\begin{align*}
C = \bigcup^{\infty}_{n=0} C_n &\Rightarrow C_0 \subseteq C\\
&\Leftrightarrow A \setminus C \subseteq A \setminus C_0 = A_1
\end{align*}

Indeed, $A \setminus C \sim g(A \setminus C) = B \setminus f(C)$* by equation (2). We know $C \subseteq A$, so by Lemma 3.1 we have $C \sim g(C)$**. 

The fact that $(A \setminus C) \cap C = \emptyset = (B \setminus f(C)) \cap f(C)$ with (*) and (**) allows us to utilise Lemma 3.2, so $(A \setminus C) \cup C \sim (B \setminus f(C)) \cup f(C) \Rightarrow A \sim B$.

\end{proof}

\begin{thm}
$\mathbb{N}\cup\{0\}$ and $\mathbb{Z} $ are G-equidecomposable, where G is the group of all bijections on $\mathbb{Q} $.
\end{thm}
\begin{proof}
Let $N_1$ be the set of even numbers (including zero), and $N_2$ be the set of odd numbers. So
\begin{equation*}
N_1 \cup N_2 = \{ 0, 2, 4, 6, 8 ....\} \cup \{1, 3, 5, 7, 9 ....\} = \mathbb{N} \cup \{ 0 \}
\end{equation*}
Let $Z_1$ be the set of positive integers (including zero) and $Z_2$ be the set of odd integers. So
\begin{equation*}
Z_1 \cup Z_2 = \{ 0, 1, 2, 3, 4 ....\} \cup \{-1, -2, -3, -4 ....\} = \mathbb{Z}
\end{equation*}
and $N_1 \cap N_2 = \emptyset = Z_1 \cap Z_2$.\\

Define the bijection $g_1: N_1 \rightarrow Z_1$ as
\begin{align*}
g_1  x &= \frac{x}{2}\\
\therefore g_1 \cdot N_1 &= Z_1, \text{ where } g_1 \in G
\end{align*}

Define the bijection $g_2: N_2 \rightarrow Z_2$ as
\begin{align*}
g_2  x &= -\frac{x + 1}{2}\\
\therefore g_2 \cdot N_2 &= Z_2, \text{ where } g_2 \in G
\end{align*}

So $\mathbb{N}\cup\{0\}$ and $\mathbb{Z} $ each can be partitioned into two respectively $G$-congruent pieces, hence $(\mathbb{N}\cup\{0\}) \sim \mathbb{Z} $
\end{proof}

\begin{thm}
$S^1$ is $SO_2$-equidecomposable to $S^1 \setminus \{ 1 \}$, where $SO_2$ = group of rotations  in $\mathbb{R}^2$, and $S^1 = \{ (x_1, x_2) \in \mathbb{R}^2 : x_1^2 + x_2^2 = 1\}$ the circle radius 1 centred about the origin.
\end{thm}
\begin{proof}
Let $\theta$ be a counterclockwise rotation by, say, 1 radians around the origin. Since $2\pi$ is irrational we see that $\theta^n \cdot 1$ never comes back to coincide with $1$ for all $n \in \mathbb{N}.$

Let
\begin{align*}
A_1 &= \{ 1, \theta \cdot 1, \theta^2 \cdot 1, \theta^3 \cdot 1, ....\} \text{ and }A_2 = S^1 \setminus A_1
\end{align*}
so that $A_1 \cup A_2 = S^1$ and $A_1 \cap A_2 = \emptyset$ 

Let
\begin{align*}
B_1 &= \theta \cdot A_1 = \{\theta \cdot 1, \theta^2 \cdot 1, \theta^3 \cdot 1, ....\}\\
 \text{ and } B_2 &= 1 \cdot A_2 = S^1 \setminus A_1 ,\text{where 1 is the identity element in }SO_2
\end{align*}
so that $B_1 \cup B_2 = S^1\setminus \{1\}$ and $B_1 \cap B_2 = \emptyset$ 

\end{proof}
This theorem illustrates how you can take a circle, partition it into no more than two pieces, rearrange one piece via the group of rotations to give you two pieces of that original circle only now with a point missing. The choice of rotation was paramount here, ensuring that no multiple of $\theta$ will be the identity rotation and hence fill the hole in the circle at 1.
\begin{thm}
$B^3$ is $SO_3$-equidecomposable with $B^3 \setminus \{ 0 \} $, where $SO_3$ is the group of rotations in $ \mathbb{R}^3$
\end{thm}
This theorem describes how a solid ball can be divided into pieces and rearranged to give the same solid ball minus a single point.
\begin{proof}
Let $C^1$ be a circle passing through the origin contained in $B^3$. Let $A_0 = B^3 \setminus C^1$ and $A = C^1$ so that $A_0$, $A$ partition $ B^3$. Let $B_0 =A_0 =B^3 \setminus C^1$ and $B= C^1 \setminus \{0\}$ so $B_0$, $B$ partition $B^3 \setminus \{0\}$. 

Theorem 3.5 gives us that $A$ and $B$ are $SO_2$-equidecomposable. Hence $A$ and $B$ can be partitioned into a finite number, say $n$, of $SO_2$-congruent pieces. i.e, 
\begin{align*}
A &= \bigcup^n_{i=1} A_i \text{ and } B = \bigcup^n_{i=1} B_i &\\
B_j &= h_j \cdot A_j &\text{for } 0 < j \le n, h_j \in SO_2
\end{align*}
Since $SO_2 \subseteq SO_3$, $h_j \in SO_2 \Rightarrow h_j \in SO_3$ for $0 < j \le n$. $B_0 = 1 \cdot A_0$, where 1 is the identity element in $SO_3$. Therefore,
\begin{align*}
B^3 = A_0 \cup (\bigcup^n_{i=1} A_i)  =\bigcup^n_{i=0} A_i
\text { and }
B^3 \setminus \{0\} &= B_0 \cup (\bigcup^n_{i=1} B_i ) = \bigcup^n_{i=0} B_i
\end{align*}
and
\begin{align*}
B_j &= h_j \cdot A_j &\text{for } 0 \le j \le n, h_j \in SO_3
\end{align*}

\end{proof}
It can now be seen that a solid ball in $\mathbb{R}^3$ can be partitioned into two pieces; one a ring passing through the origin, and the other being the remainder of the ball. The ring, as seen before, is equidecomposable to the ring minus a point. So using that phenomenon on a ring, as a piece of the ball, yields the same result in a higher dimension.

\section{The importance of independence}
A set $S$ of elements in a group are called \emph{independent} if no non-trivial, reduced words using letters from $S$ and their inverses is the identity. Hence, a pair of independent elements will generate a free subgroup or rank 2.

\begin{thm}
There are two independent rotations about the axis through the origin in $\mathbb{R}^3$. Hence, $SO_3$ contains a free subgroup of rank 2. 
\end{thm}

\begin{proof} Let $\sigma$ be a counter clockwise rotation about the z-axis through an angle of $\theta = cos^{-1} \left (\frac {3}{5} \right)$. Let $\tau$ be a counter clockwise rotation about the x-axis through the same angle $\theta$.

We have\\
\[\sigma =  \left( \begin{array}{ccc}
\frac {3}{5}  & \frac {-4}{5}  & 0 \\
\frac {4}{5}  & \frac {3}{5}  & 0 \\
0 & 0 & 1 \end{array} \right)\] 

and\\
\[\tau =  \left( \begin{array}{ccc}
 1 & 0 & 0 \\
0 & \frac {3}{5}  &  \frac {-4}{5} \\
0 & \frac {4}{5}  & \frac {3}{5} \end{array} \right)\] 

We wish to show that no non-trivial reduced word in $\{ \sigma, \sigma^{-1}, \tau, \tau^{-1} \}$ is the identity homeomorphism. 
We claim that conjugation by $\sigma$ does not effect whether or not a word is the identity. To see this, let $y$ be some non-trivial, reduced word in $\{ \sigma, \sigma^{-1}, \tau, \tau^{-1} \}$ and let $\gamma = \sigma^{-1} \cdot y \cdot \sigma$. There are two cases;
\begin{itemize}
\item[(i)] $y$ contains either $\tau$ or $\tau^{-1}$. Then $\gamma$ must be non-trivial when reduced.
\item[(ii)] $y$ does not contain either $\tau$ or $\tau^{-1}$. Then, since $y$ was assumed to be reduced, $y$ can only be either a sequence of $\sigma$ or a sequence of $\sigma^{-1}$ of length $>$ 0. No complete reduction can occur in $\gamma$ (leaving $\gamma$ non-trivial when reduced), hence our claim is true.
\end{itemize}

In order to prove the theorem, we need only to consider words ending (on the right) with $\sigma^{\pm 1}$ (words ending in $\sigma^{-1}$ after conjugation by $\sigma$ for the case when $y=...\sigma^{-1} \cdot \sigma^{-1}$, or $y = \sigma^{-1}$). We now claim that for any such word $\omega$ the point $\omega ( 1, 0, 0) $  will take the form $(a, b, c)/5^{n}$, where n is the integer length of $\omega$; $a,b$ and $c$ are all integers with 
$b$ not divisible by 5. Given that zero is indeed divisible by 5, this claim implies that $\omega ( 1, 0, 0) \not= ( 1, 0, 0) $, so $\omega$ is not the identity homeomorphism.\\

Let us prove this claim by induction;\\
For $n = 1$, we have $\omega = \sigma^{\pm 1}$ and so
\begin{align*}
\omega(1,0,0) &= \sigma^{\pm 1}(1,0,0) \\
& = \frac {1}{5}  \left (\begin{array}{ccc}
3& \mp 4 & 0 \\
\pm 4  & 3  & 0 \\
0 & 0 & 5 \end{array} \right)
\left (\begin{array}{c}
1\\
0\\
0\end{array} \right)\\
& = (3, \pm 4, 0) / 5
\end{align*}
which is of the required form.

Fix $k \in \mathbb{N} \setminus \{1\}$. Assume claim is true for $n = 1, 2, .., k-1$. Let $\omega_n$ be any non-trivial, reduced word ending in $\sigma^{\pm 1}$. By assumption, $\omega_n (1,0,0) = (a_n, b_n, c_n) / 5^n$, where $a_n, b_n, c_n \in \mathbb{Z}$ and $b_n$ not divisible by 5.

Write $\omega_k = \phi \omega_{k-1}$ for any reduced word $\omega_{k}$, where $\mid \omega_{k} \mid = k$ and $\phi \in \{\sigma, \sigma^{-1}, \tau, \tau^{-1} \}$.There are four cases:
\begin{itemize}
\item[Case 1:] $\omega_k = \sigma \omega_{k-1}$
\begin{align*}
\omega_k (1,0,0) &= \sigma \omega_{k-1}(1,0,0)\\
& =\sigma(a_{k-1}, b_{k-1}, c_{k-1})/5^{k-1}\\
& = \frac {1}{5^k}  \left (\begin{array}{c}
{3a_{k-1} - 4b_{k-1}}\\
{4a_{k-1} + 3b_{k-1}} \\
{5c_{k-1}} \end{array} \right)\\
\end{align*}

We can see that $a_k = 3a_{k-1} - 4b_{k-1}$ ,  $b_k = 4a_{k-1} + 3b_{k-1}$ , $c_k = 5c_{k-1}$ are all integers. We must check now to see if 
\begin{equation}
b_k = 4a_{k-1} + 3b_{k-1}
\end{equation} is divisible by 5. To do this, we need to look now at 3 more cases;  $ \omega_k = \sigma \tau^{\pm1} \omega_{k-2}$ , $\omega_k = \sigma \sigma \omega_{k-2}$, where $\omega_{k-2}$ is possibly the empty word. 

For the first two cases, we already know that $b_{k-1}$ is not  divisible by 5, so what we really want to check is if $a_{k-1}$ is divisible by 5, for then the addition of the two terms will yield a number not divisible by 5. We have
\begin{align*}
\omega_k (1,0,0) &= \sigma \tau^{\pm1} \omega_{k-2}(1,0,0)\\
&= \sigma \frac {1}{5^{k-2}}  \times \frac {1}{5} \left (\begin{array}{ccc}
5 & 0 & 0 \\
0 & 3 & \mp4\\
0 & \pm4 & 3 \end{array} \right)
\left (\begin{array}{c}
{a_{k-2}}\\
{b_{k-2}}\\
{c_{k-2}}\end{array} \right)\\
&= \sigma \frac {1}{5^{k-1}}  \left (\begin{array}{c}
{5a_{k-2}}\\
{3b_{k-2} \mp 4c_{k-2}}\\
{4b_{k-2}\pm3c_{k-2}} \end{array} \right)
\end{align*}

Since $a_{k-1} = 5a_{k-2}$, $a_{k-1}$ is divisible by 5 when $ \omega_k = \sigma \tau^{\pm1} \omega_{k-1}$. 

For the second case, we check to see if $b_k$ is divisible by 5 based on the assumption that both $b_{k-1}$ and $b_{k-2}$ are not divisible by 5. We have
\begin{align*}
\omega_k (1,0,0) &= \sigma \sigma \omega_{k-2}(1,0,0)\\
&= \sigma \frac {1}{5^{k-2}}  \times \frac {1}{5} \left (\begin{array}{ccc}
3 & -4 & 0 \\
4 & 3 & 0\\
0 & 0 & 5 \end{array} \right)
\left (\begin{array}{c}
{a_{k-2}}\\
{b_{k-2}}\\
{c_{k-2}}\end{array} \right)\\
&= \sigma \frac {1}{5^{k-1}}  \left (\begin{array}{c}
{3a_{k-2} - 4b_{k-2}}\\
{4a_{k-2} +3b_{k-2}} \\
{5c_{k-2}}\end{array} \right)
\end{align*}

So now,
\begin{align}
a_{k-1} &= 3a_{k-2} - 4b_{k-2}\\
b_{k-1} &= 4a_{k-2} + 3b_{k-2} \Rightarrow 4a_{k-2} = b_{k-1} -  3b_{k-2}
\end{align}

Subbing equations (3) and (4) into (2) gives
\begin{align*}
b_{k} &= 4( 3a_{k-2} - 4b_{k-2}) + 3b_{k-1}\\
&= 3(b_{k-1} -  3b_{k-2}) -16b_{k-2} + 3b_{k-1}\\
&= 6b_{k-1} - 25b_{k-2}
\end{align*}
Hence $b_k$ is not divisible by 5 for Case 1.

\item[Case 2:] $\omega_k = \sigma^{-1} \omega_{k-1}$
\begin{align*}
\omega_k (1,0,0) &= \sigma^{-1} \omega_{k-1}(1,0,0)\\
& = \frac {1}{5^k}  \left (\begin{array}{c}
{3a_{k-1} +4b_{k-1}}\\
{-4a_{k-1} + 3b_{k-1}} \\
{5c_{k-1}} \end{array} \right)\\
\end{align*}

Again, we see that $a_k = 3a_{k-1} + 4b_{k-1}$ ,  $b_k = -4a_{k-1} + 3b_{k-1}$ , $c_k = 5c_{k-1}$ are all integers. We must check now to see if 
\begin{equation}
b_k = -4a_{k-1} + 3b_{k-1}
\end{equation} is divisible by 5. So we take the cases 
$\omega_k = \sigma^{-1} \tau^{\pm1} \omega_{k-2}$ and $\omega_k = \sigma^{-1}\sigma^{-1} \omega_{k-2}$. 
The same working from Case 1 applies here in that $a_{k-1} = 5a_{k-2}$, so $a_{k-1}$ is divisible by 5 when $ \omega_k = \sigma^{-1} \tau^{\pm1} \omega_{k-1}$. Now we have
\begin{align*}
\omega_k (1,0,0) &= \sigma^{-1} \sigma^{-1}  \omega_{k-2}(1,0,0)\\
&= \sigma^{-1}  \frac {1}{5^{k-1}}  \left (\begin{array}{c}
{3a_{k-2} + 4b_{k-2}}\\
{-4a_{k-2} +3b_{k-2}} \\
{5c_{k-2}}\end{array} \right)
\end{align*}

So now,
\begin{align}
a_{k-1} &= 3a_{k-2} + 4b_{k-2}\\
b_{k-1} &= -4a_{k-2} + 3b_{k-2} \Rightarrow 4a_{k-2} = 3b_{k-2} - b_{k-1}
\end{align}

Subbing equations (6) and (7) into (5) gives
\begin{align*}
b_{k} &= -4( 3a_{k-2} + 4b_{k-2}) + 3b_{k-1}\\
&= -3( 3b_{k-2} - b_{k-1}) -16b_{k-2} + 3b_{k-1}\\
&= 6b_{k-1} - 25b_{k-2}
\end{align*}
Hence $b_k$ is not divisible by 5 for Case 2.

\item[Case 3:] $\omega_k = \tau \omega_{k-1}$
\item[Case 4:] $\omega_k = \tau^{-1} \omega_{k-1}$
\end{itemize}
A similar proof can be used to show that $b_k$ is not divisible by 5 for Case 3 and Case 4, and that $a_k$, $b_k$, $c_k$ are all integer values.

\end{proof}

\begin{defi} A set $S$ is \emph{countable} if it has the same cardinality as a subset of $\mathbb{N}$. i.e, if $\mid S \mid = \mid N \mid$, $N \subseteq \mathbb{N}$
\end{defi}

\begin{lem} Let $D = \{ x \in S^2: \text{ x is fixed by some element in }F_2\setminus\{e\} \}$. Then $D$ is countable and $S^2$ is $SO_3$-equidecomposable with $S^2\setminus D$.
\end{lem}

\begin{proof}
Each non-identity rotation in $F_2$ fixes precisely two points on $S^2$, namely the intersection of the axis of rotation with the sphere. $D$ is the collection of such points. 

Let $D_n$ denote the points on $S^2$ fixed by words in $F_2\setminus \{e\}$ of length $n \in \mathbb{N}$.
\begin{align*}
\mid D_1 \mid = & 4 \\ 
\mid D_2 \mid = &4 \times 3 \\ 
\mid D_3 \mid = &4  \times 3^2 \\ 
&\vdots \\ 
\mid D_n \mid = &4 \times 3^{n-1} 
\end{align*}

Then $D=\bigcup^{\infty}_{n=1}D_n$ which is a countable union of finite sets, hence $D$ is countable.

We claim that $S^2$ is not countable. Define a bijection $f$ as
\begin{align*}
f: S^1 &\rightarrow [0,1) \\
(cos\theta, sin\theta) &\mapsto \frac{\theta}{2\pi}
\end{align*}

So $\mid S^1 \mid = \mid[0,1)\mid \ge \mid \mathbb{N} \mid$, given that the interval [0,1) is uncountable. Hence $S^1$ is uncountable. Then, since $S^1 \subseteq S^2$, $S^2$ is uncountable and we have proved our claim.

It follows, since $D$ is countable, and $S^2$ is uncountable, that $S^2 \setminus D$ is non-empty and uncountable. Hence there exists a pair of antipodal points on $S^2$ which aren't elements of $D$. Let $\ell$ be the axis through the centre and these points. We wish to find a rotation $\mu$ about the axis $\ell$ such that $D$, $\mu D$, $\mu^2 D$ ... are disjoint.

Let $A= \{ \rho : \rho \text{ is a rotation about }\ell \text{ and }\rho^mD \bigcap_{m \in \mathbb{N}} D \not= \emptyset \}$.
Now, $\rho^mD \cap D \not= \emptyset \Rightarrow \exists d_1, d_2 \in D : \rho^m d_1=d_2$. Since $d_1$ and $ d_2$ are part of a countable set, then $A$ is countable. Choose a rotation $\mu$ about $\ell$ that is not in $A$. So
\begin{align*}
\mu^{m-n}D\cap D = \emptyset \Rightarrow \mu^{m} D \cap \mu^nD = \emptyset 
\text{ for all }m,n \in \mathbb{N}
\end{align*}
as required.

Let $A_1 = D \cup \rho D \cup \rho^2D \cup ...$ and $A_2 = S^2 \setminus A_1$ so that $A_1 \cap A_2 = \emptyset$ and $A_1 \cup A_2 = S_2$.
Let $B_1 = \rho A_1 =  \rho D \cup \rho^2D \cup ...$ and $B_2 = A_2$ so that $B_1 \cap B_2 = \emptyset$ and $B_1 \cup B_2 = S_2 \setminus D$. We have partitioned $S_2$ and $S_2 \setminus D$ into two $SO_3$-congruent pieces, hence $S_2 \sim S_2 \setminus D$.

\end{proof}

\section{The Banach-Tarski Paradox}
As mentioned in the Abstract, the Banach-Tarski Paradox illustrates how one can cut a solid ball in $\mathbb{R}^3$ into finitely many pieces, where each piece can only be rotated and translated to form two balls identical to the original. 

In proving this, most of the work will be in proving the phenomenon for a sphere in $\mathbb{R}^3$ where the radius of the sphere is irrelevant. So we will endeavour to partition a sphere into two pieces, each of which can be rearranged to form the original sphere. Once we have this we can form a ball (minus the origin) by the infinite union of spheres.

\begin{lem}
Let $x\in S^2 \setminus D$ and $f,g \in \mathbb{F}_2$. Then 
\begin{equation*}
f \cdot x = g \cdot x\Rightarrow f=g
\end{equation*}
\end{lem}	
\begin{proof}
\begin{align*}
D &= \{ x \in S^2 : x\text{ is fixed by some point in }\mathbb{F}_2 \setminus \{e\} \} \\
S^2 \setminus D &= \{ x \in S^2 : x\text{ is fixed only by }\{e\} \} 
\end{align*}
So when $f \cdot x = g \cdot x \Rightarrow g^{-1}f \cdot x = x$, $x$ is fixed by $g^{-1}f$. Hence $g^{-1}f = e \Rightarrow g=f$.
\end{proof}	
	
\begin{lem}
Let $x\in S^2$ and $f \in \mathbb{F}_2$. Then
\begin{equation*}
x \in D \Leftrightarrow f \cdot x \in D
\end{equation*}
\end{lem}
	
\begin{proof} ($\Leftarrow$)
\begin{align*}
f \cdot x \in D &\Rightarrow \exists g \in \mathbb{F}_2 \setminus \{e\} \text{ with }gf\cdot x = f \cdot x \\
&\Rightarrow f^{-1}gf \cdot x = x
\end{align*}
In order to show $x\in D$, since $x$ is fixed by $f^{-1}gf$ we are only required to show that $f^{-1}gf \in \mathbb{F}_2 \setminus \{e\} $. If $f = e$, then $f^{-1}gf = g \in \mathbb{F}_2 \setminus \{e\}$. If $f \not= e$, then, since $g$ is a non-trivial word in the free group  $\mathbb{F}_2$, conjugation by $f$ does not affect whether or not a word is the identity (conjugation result discussed in Theorem 4.1);
 $f^{-1}gf \not= e$. So $x\in D$.\\
($\Rightarrow$) 
\begin{align*}
x \in D &\Rightarrow \exists h \in \mathbb{F}_2 \setminus \{e\} \text{ with } h \cdot x =x \\
&\Rightarrow (fhf^{-1})f \cdot x = f(h \cdot x) =f \cdot x
\end{align*}
Similarly, if $f=e$, then $fhf^{-1} = h \in \mathbb{F}_2 \setminus \{e\}$. If $f \not= e$, then
 $f^{-1}hf \not= e$. So $f \cdot x\in D$
\end{proof}	
	
\begin{defi} An \emph{$\mathbb{F}_2$ - orbit} is a set of the form
\begin{equation*}
F_x := \{ f \cdot x : f \in \mathbb{F}_2 \} \text{ for some } x \in S^2 \setminus D
\end{equation*}
Denote a family of $\mathbb{F}_2$-orbits by $(F_i)_{i\in I}$ for some interval $I$.
\end{defi}
	
\begin{lem}
Two $\mathbb{F}_2$-orbits are either disjoint or equal
\end{lem}

\begin{proof}
Assume two $\mathbb{F}_2$-orbits say, $F_x$ with $x \in S^2 \setminus D$ and $F_y$ with $y \in S^2 \setminus D$, are not disjoint. So $F_x \cap F_y \not= \emptyset$. Fix
\begin{equation}
w \in F_x \cap F_y \Rightarrow w = g \cdot x = h \cdot y \text{ for some }g,h \in \mathbb{F}_2
\end{equation}

We wish to show that $F_x = F_y$. 
\begin{align*}
\text{Fix }z \in F_x  \Rightarrow& z = f \cdot x  \text{ for some }f \in \mathbb{F}_2 \\
\Rightarrow& f^{-1} \cdot z = x \\
\text{subbing into (9) gives } & g(f^{-1} \cdot z) = h \cdot y \\
\Rightarrow& z = (gf^{-1})^{-1} h \cdot y \in F_y
\end{align*}
Therefore,  $F_x \subseteq F_y$.
\begin{align*}
\text{Fix }z \in F_y  \Rightarrow& z = f \cdot y  \text{ for some }f \in \mathbb{F}_2 \\
\Rightarrow& f^{-1} \cdot z = y \\
\text{subbing into (9) gives } & h(f^{-1} \cdot z) = g \cdot x \\
\Rightarrow& z = (hf^{-1})^{-1} g \cdot x \in F_x
\end{align*}
Therefore, $F_y \subseteq F_x$ and $F_y = F_x$.

The other direction of the proof was to show that two disjoint $\mathbb{F}_2$-orbits are not equal, which is trivial.
\end{proof}	
The implication of this Lemma is that if two $\mathbb{F}_2$- orbits have a single point in common, then the orbits are equivalent. 

\begin{lem}
$\mathbb{F}_2$-orbits in $S^2 \setminus D$ union to give $S^2 \setminus D$
\end{lem}

\begin{proof}
$F_x = \{ f \cdot x: f \in \mathbb{F}_2 \}$ for some $x \in S^2 \setminus D$. So $f \cdot x \in S^2 \setminus D$ by Lemma 5.2. Since $x$ is arbitrary, $\bigcup_{i \in I} F_i \subseteq S^2 \setminus D$
\begin{align*}
\text{Now fix }x \in S^2 \setminus D &\Rightarrow F_x = \{ f \cdot x: f \in \mathbb{F}_2 \}\text{ is the $\mathbb{F}_2$-orbit}\\
&\Rightarrow F_x = F_i \text{ for some } i \in I \text{ by Lemma 5.4}\\
&\Rightarrow x \in F_x \text{ because } e\in  \mathbb{F}_2 
\end{align*}
So $x \in F_x \subseteq \bigcup_{i \in I} F_i$. Hence, $\bigcup_{i \in I} F_i = S^2 \setminus D$.
\end{proof}

\subsection{Revisiting $\mathbb{F}_2$; Free group of rank 2}
\begin{lem}
There exists subsets $P_1, P_2 \subseteq \mathbb{F}_2$ where $P_1 \sim_{\mathbb{F}_2} \mathbb{F}_2$ and $P_2 \sim_{\mathbb{F}_2} \mathbb{F}_2$, and $P_1, P_2$ partition $ \mathbb{F}_2$.
\end{lem}
We can divide $\mathbb{F}_2$ into two sets, each of which has pieces that can be acted upon in such a way so that each set forms $\mathbb{F}_2$ again. We began looking at free groups of rank 2 in Theorem 1.4, although a similar concept, this lemma is a much stronger version. Let us adopt the notation defined previously. 

\begin{proof}
\begin{align*}
A_1 &= \Psi( a^{-1}  ) \cup \{e\} \cup \{a\} \cup \{a^2\} \cup \{a^3\}... \\
A_2 &= \Psi(a) \setminus \{e\} \cup \{a\} \cup \{a^2\} \cup \{a^3\}...\\
B_1 &= \Psi(b^{-1} )\\
B_2 &= \Psi(b)
\end{align*}
Let $P_1 = A_1 \cup A_2$, $P_2 = B_1 \cup B_2$. $A_1$ and $A_2$ are disjoint so clearly $A_1$, $A_2$ partition $P_2$.  As previously mentioned in Theorem 1.4, $(a \cdot A_1) \cup A_2 = \mathbb{F}_2$. Hence $P_1 \sim_{\mathbb{F}_2} \mathbb{F}_2$. Similarly $B_1$, $B_2$ partition $P_2$ and $b \cdot B_1 \cup B_2 = \mathbb{F}_2$. Hence $P_2 \sim_{\mathbb{F}_2} \mathbb{F}_2$. Also, 
\begin{align*}
(A_1 \cup A_2) \cup (B_1 \cup B_2) =  \mathbb{F}_2 &\Rightarrow P_1 \cup P_2 =  \mathbb{F}_2 \\
\text{and }(A_1 \cap A_2) \cap (B_1 \cap B_2) =  \emptyset &\Rightarrow P_1 \cup P_2 = \emptyset
\end{align*}
So $P_1, P_2$ partition $ \mathbb{F}_2$.
\end{proof}

\subsection{Tying it all together}
\begin{thm}[Banach-Tarski Paradox with spheres]
There exists subsets $S_1, S_2 \subseteq S^2$ that partition $S^2$ and $S_1 \sim_{SO_3} S^2$, $S_2 \sim_{SO_3} S^2$
\end{thm}

We wish to find a set $M \subseteq S^2$ through which we act on to translate the equidecomposibility properties of $\mathbb{F}_2$ over to $S^2$. By the Axiom of Choice, 
$\exists M \subseteq S^2 \setminus D\text{ such that }\forall i \in I$, $| M \cap F_i | = 1$.
Plainly put, $M$ is manifested by taking one element from each $F$-orbit, $F_i$.

\begin{lem}
For $P_1, P_2 \subseteq \mathbb{F}_2$ and $M \subseteq S^2 \setminus D$, define
$P_1 M = \{ f \cdot m : f \in P_1, m\in M \}$, $P_2 M = \{ f \cdot m : f \in P_2, m\in M \}$. Then, $P_1 M$, $P_2 M \subseteq S^2 \setminus D$ and 
\begin{align*}
P_1M \cup P_2 M &= (P_1 \cup P_2) M \\
P_1M \cap P_2 M &= (P_1 \cap P_2) M 
\end{align*}
\end{lem}

\begin{proof}
\begin{align*}
f \in P_1M \cup P_2 M \Leftrightarrow& f \in \{ \{f \cdot m : f \in P_1, m\in M \} \cup \{ f \cdot m : f \in P_2, m\in M \} \} \\
\Leftrightarrow& f \in \{f \cdot m : f \in P_1 \text{ or }f \in P_2, m\in M \} \\
\Leftrightarrow& f \in \{f \cdot m : f \in P_1\cup P_2, m\in M \} \\
\Leftrightarrow& f \in (P_1 \cup P_2) M 
\end{align*}
Similar working shows that $f \in P_1M \cap P_2 M \Leftrightarrow f \in (P_1 \cap P_2) M $.

Lemma 5.2 gives for $m \in M \subseteq (S^2 \setminus D) \subseteq S^2$ and $f \in P_1 \subseteq \mathbb{F}_2$, that $m \in S^2 \setminus D \Leftrightarrow f\cdot m \in S^2 \setminus D$. Hence, $P_1 M = \{ f \cdot m : f \in P_1, m\in M \} \subseteq S^2 \setminus D$. Similarly, $P_2 M\subseteq S^2 \setminus D$.
\end{proof}

\begin{proof}[Proof of the Banach-Tarski paradox with spheres]
It will be shown that, because $P_1$ and $P_2$ union to give $\mathbb{F}_2$, then $P_1 M$ and $P_2 M$ union to give $\mathbb{F}_2 M$.
\begin{align*}
P_1M \cup P_2 M &= (P_1 \cup P_2) M &\text{(Lemma 5.7)}\\
&= \mathbb{F}_2 M \\
&= \{ f \cdot m : f \in \mathbb{F}_2, m\in M \subseteq (S^2 \setminus D)  \} \\
&= \text{union of all $F$-orbits in $S^2 \setminus D$} \\
&= S^2 \setminus D &\text{(Lemma 5.5)}
\end{align*}
Also, since $P_1$ and $P_2$ are disjoint in $\mathbb{F}_2$, then $P_1 M$ and $P_2 M$ are disjoint in $S^2$
\begin{align*}
P_1M \cap P_2 M &= (P_1 \cap P_2) M &\text{(Lemma 5.7)}\\
&= \emptyset M \\
&= \emptyset
\end{align*}

\begin{lem}
$P_1 M \sim_{SO_3} \mathbb{F}_2 M$ and $P_2 M \sim_{SO_3} \mathbb{F}_2 M$
\end{lem}
\begin{proof}
Using $A_1$, $A_2$, $B_1$, $B_2$ from the proof of Lemma 5.6, we know that $A_1 \cup A_2 = P_1$, $A_1 \cap A_2 = \emptyset$ and $a \cdot A_1 \cup A_2 = \mathbb{F}_2$. So\\
$A_1M \cup A_2M = (A_1 \cup A_2)M = P_1 M $\\
$A_1M \cap A_2M = (A_1 \cap A_2)M = \emptyset $\\
and, $a \cdot A_1M \cup A_2M = (a \cdot A_1 \cup A_2) M = \mathbb{F}_2 M$. Hence $P_1 M \sim_{SO_3} \mathbb{F}_2 M$. \\

We also know that $B_1 \cup B_2 = P_2$, $B_1 \cap B_2 = \emptyset$ and $b \cdot B_1 \cup B_2 = \mathbb{F}_2$. So \\
$B_1M \cup B_2M = (B_1 \cup B_2)M = P_2 M $\\
$B_1M \cap B_2M = (B_1 \cap B_2)M = \emptyset $\\
and similarly, $b \cdot B_1M \cup B_2M = (b \cdot B_1 \cup B_2) M = \mathbb{F}_2 M$. Hence $P_2 M \sim_{SO_3} \mathbb{F}_2 M$ \\
\end{proof}

We have used the fact that $P_1$ and $P_2$ are $\mathbb{F}_2$-equidecomposable to $\mathbb{F}_2$ to show that $P_1 M$ and $P_2 M$ are both $SO_3$-equidecomposable to $\mathbb{F}_2 M = S^2 \setminus D$. \\

\emph{Proof of the Banach-Tarski paradox with spheres continued:}
Define $S_1 := P_1 M \cup D$ and $S_2 := P_2 M$. We need to show that $S_1, S_2$ partition $S^2$.
\begin{align*}
S_1 \cup S_2 =& (P_1 M \cup D) \cup P_2 M \\
=& P_1 M \cup P_2 M \cup D  \\
=& \mathbb{F}_2 M \cup D \\
=& S^2 \setminus D \cup D \\
=& S^2
\end{align*}

\begin{align*}
S_1 \cap S_2 =& (P_1 M \cup D) \cap P_2 M \\
=& (P_1M \cap P_2M) \cup (D \cap P_2M) &\text{ (Lemma 5.8)}\\
=& \emptyset
\end{align*}

Now, $S_1 = P_1 M \cup D$, and we know from Lemma 5.9 that $P_1 M \sim \mathbb{F}_2 M$ and  $D \sim D$ is trivial to see. So Lemma 3.2 gives $(P_1 M \cup D) \sim (\mathbb{F}_2 M \cup D)$. We have
\begin{align*}
S_1 \sim& \mathbb{F}_2 M \cup D \\
=& S^2 \setminus D \cup D \\
=& S^2
\end{align*}
We also have
\begin{align*}
S_2 =& P_2 M \\
\sim& \mathbb{F}_2 M &\text{(Lemma 5.9)}\\
=& S^2 \setminus D \\
\sim& S^2 &\text{ (Lemma 4.3)}
\end{align*}
\end{proof}

\begin{thm}[Banach-Tarski paradox]
There exists subsets $\mathcal{B}_1, \mathcal{B}_2 \subseteq B^3$ that partition $B^3$ and $\mathcal{B}_1 \sim_{SO_3} B^3$, $\mathcal{B}_2 \sim_{SO_3} B^3$
\end{thm}

\begin{proof}[Proof of the Banach-Tarski paradox]

There are 5 pieces of the sphere;  $A_1M$, $A_2M$, $B_1M$, $B_2M$, $D$. Remembering that these pieces are grouped accordingly so  $S_1 =  A_1M \cup A_2M \cup D$, $S_2 = B_1M \cup B_2M$. Fix $0 < r \le 1$. Define $rS_i = \{ rx | x \in S_i \}$ for $i = 1,2$. 

Define $\mathcal{B}_1 = \bigcup_{0 < r \le 1} rS_1$ and $\mathcal{B}_2 = \bigcup_{0 < r \le 1} rS_2$ so that $\mathcal{B}_1, \mathcal{B}_2 \subseteq B^3$. Since $S_1$ and $S_2$ are disjoint, then by construction $\mathcal{B}_1$ and $\mathcal{B}_2$ are disjoint. To see that $\mathcal{B}_1$ and $\mathcal{B}_2$ partition $B^3 \setminus \{0\}$ note that
\begin{align*}
\mathcal{B}_1 \cup \mathcal{B}_2 &= \bigcup_{0 < r \le 1} r(S_1 \cup S_2) \\
&=\bigcup_{0 < r \le 1} rS^2 \\
&=B^3 \setminus \{0\}
\end{align*}

The three pieces of the ball that partition $\mathcal{B}_1$ are $(\bigcup_{0 < r \le 1} rA_1M)$, $ (\bigcup_{0 < r \le 1} rA_2M)$ and $ (\bigcup_{0 < r \le 1} rD)$. They can be rearranged as follows; $(\bigcup_{0 < r \le 1} r(a \cdot A_1)M)$, $ (\bigcup_{0 < r \le 1} rA_2M)$ and $ (\bigcup_{0 < r \le 1} rD)$. It can be seen that these rearranged pieces partition $B^3 \setminus \{0\}$. Hence $\mathcal{B}_1 \sim_{SO_3} B^3 \setminus \{0\}$. Recall from Theorem 3.6 that $B^3 \setminus \{0\} \sim_{SO_3} B^3$, so $\mathcal{B}_1 \sim_{SO_3} B^3$.

Two pieces of the ball, $(\bigcup_{0 < r \le 1} rB_1M)$ and $ (\bigcup_{0 < r \le 1} rB_2M)$ partition $\mathcal{B}_2$. When the pieces are rearranged to form $(\bigcup_{0 < r \le 1} r(b \cdot B_1)M)$ and $ (\bigcup_{0 < r \le 1} rB_2M)$, then these pieces partition $B^3\setminus \{0\}$. Hence  $\mathcal{B}_1 \sim_{SO_3} B^3 \setminus \{0\} \sim_{SO_3} B^3$.
\end{proof}

\section*{Consequences of the Banach-Tarski Paradox}

The corresponding paradox for free groups, while somewhat surprising, is certainly not something that anyone would argue with.  When this paradox is applied to 3- dimensional space it does go against our intuition, but very often our intuition is flawed. 

By allowing the use of the Axiom of Choice, we can obtain sets which are so complicated that they cannot be assigned a measure consistent with the desired properties. This is precisely what happens in the Banach-Tarski paradox; most of the sets we have constructed are \emph{nonmeasurable};  not Lebesgue measurable. So the Banach-Tarski Paradox doesn't double the volume of the ball because the pieces in the decomposition cannot be assigned a volume - such sets would be far more intricate than the atomic structure of matter, and as such, could never be realised in real life. Such beauty exists only in the realm of mathematics.

\newpage
\section*{References}
[1] Stan Wagon. The Banach-Tarski Paradox. Cambridge University Press, 1985.

\end{document}